\documentclass[11pt]{amsart}

\usepackage[T1]{fontenc}
\usepackage[utf8]{inputenc}
\usepackage{amsmath,amssymb,amsfonts,mathtools}
\usepackage{enumitem}
\usepackage[hidelinks]{hyperref}
\usepackage{mathrsfs,a4wide}

\newtheorem{theorem}{Theorem}[section]
\newtheorem{lemma}[theorem]{Lemma}
\newtheorem{proposition}[theorem]{Proposition}
\newtheorem{corollary}[theorem]{Corollary}

\theoremstyle{definition}
\newtheorem{definition}[theorem]{Definition}
\newtheorem{remark}[theorem]{Remark}
\newtheorem{example}[theorem]{Example}
\newtheorem{problem}[theorem]{Problem}

\newcommand{\join}{\vee}
\newcommand{\meet}{\wedge}
\newcommand{\Pow}{\mathcal{P}}
\newcommand{\op}{\mathrm{op}}
\newcommand{\NOne}{\ensuremath{\mathrm{(N1)}}}
\newcommand{\NTwo}{\ensuremath{\mathrm{(N2)}}}
\newcommand{\NThree}{\ensuremath{\mathrm{(N3)}}}
\newcommand{\NFourJoin}{\ensuremath{\mathrm{(N4_{\vee})}}}
\newcommand{\NFourMeet}{\ensuremath{\mathrm{(N4_{\wedge})}}}

\begin{document}

\title{Approximate homomorphisms\\ on orthomodular lattices}

\author[T.~Kania]{Tomasz Kania}
\address[T.~Kania]{Mathematical Institute\\Czech Academy of Sciences\\\v{Z}itn\'{a} 25\\115 67 Praha 1\\Czech Republic and Institute of Mathematics and Computer Science\\ Jagiellonian University\\ {\L}ojasiewicza 6, 30-348 Krak\'{o}w, Poland}
\email{kania@math.cas.cz, tomasz.marcin.kania@gmail.com}

\thanks{IM CAS (RVO 67985840).}


\begin{abstract}
The stability programme initiated by Ulam asks when approximate solutions to algebraic
identities must lie near exact ones.
For lattices, this leads to the question of when a map that nearly preserves joins and meets
can be approximated by a genuine lattice homomorphism.
Badora--Kochanek--Przebieracz developed a neighbourhood-based framework for distributive
lattices, centred on a separation (sandwich) lemma that constructs an exact join homomorphism
between a join-subhomomorphism and a join-superhomomorphism via an order envelope.

We revisit this mechanism and identify the single step at which distributivity is used: a
decomposition identity for elements lying below a join.
Without distributivity, separation can fail already in the modular lattice $M_3$ and in a small
finite orthomodular lattice.
On the positive side, we show that separation holds in arbitrary lattices whenever the lower
bounding map is isotone.

For orthomodular lattices---algebraic models of quantum logic---we develop a blockwise stability
theory on Boolean blocks.
Approximate identities on compatible pairs yield exact homomorphic selections on each block
(for joins, for meets, and for both operations under bi-admissibility).
We present several gluing criteria for assembling blockwise selections, and we give a concrete
finite example showing that gluing can fail when block overlaps are non-trivial.
Finally, in the spirit of Kalton--Roberts, we obtain blockwise approximation results for nearly
additive functions on orthomodular lattices by finitely additive measures, with an illustration
on finite-dimensional projection lattices.
\end{abstract}

\maketitle

\section{Introduction}

The stability programme initiated by Ulam asks when a mathematical object that satisfies
an algebraic law only approximately must lie close to an object satisfying that law exactly.
In the context of lattices, this leads to the following general problem:
given a map that nearly preserves joins and meets, under what conditions can it be
approximated by a genuine lattice homomorphism?

There are two natural and complementary ways to formalise ``approximate preservation''.
One may impose explicit error bounds, often metric in nature, as in classical Ulam--Hyers
type inequalities.
Alternatively, one may replace metric control by a system of neighbourhoods in the codomain
and express approximation through membership in these neighbourhoods.
The latter approach has the advantage of being purely order-theoretic and applies in settings
where no canonical metric is available.

A systematic neighbourhood-based framework for approximate lattice homomorphisms was
developed by Badora, Kochanek, and Przebieracz~\cite{BKP} in the distributive setting.
Their results are built around a separation (or ``sandwich'') principle:
given a join-subhomomorphism $\Phi$ and a join-superhomomorphism $\Psi$ with $\Phi\le \Psi$,
one constructs an exact join homomorphism $F$ lying between them.
The construction proceeds via an order envelope,
\[
F(x)=\sup\{\Phi(z): z\le x\},
\]
and underlies both their error-function stability results and their neighbourhood-selection
theorem.

A classical and largely independent source of motivation comes from the theorem of
Kalton and Roberts~\cite{KR} on nearly additive set functions on Boolean algebras.
That result, together with later refinements~\cite{BPR} and quantitative lower bounds~\cite{GGK},
plays a central role in the stability theory of vector measures and in the structure of
quasi-Banach spaces.
From a conceptual point of view, both the Kalton--Roberts theorem and the BKP separation
principle can be viewed as instances of a common theme: \emph{bounded defect} phenomena,
where approximate algebraic identities can be rectified by exact ones.
This viewpoint is also familiar from the theory of quasi-morphisms in geometric group theory,
where bounded cohomology provides obstructions to stability.

The first objective of the present paper is to analyse the separation mechanism of~\cite{BKP}
from a structural point of view.
A careful examination of the proof shows that distributivity of the domain lattice is used
at exactly one place: to decompose an arbitrary element $z\le x\vee y$ as
\[
z=(z\wedge x)\vee(z\wedge y).
\]
All remaining steps of the argument are purely order-theoretic and rely only on conditional
completeness of the codomain.
This observation naturally raises two questions.
First, does separation necessarily fail outside the distributive category?
Second, can distributivity be replaced by a weaker hypothesis?

We answer both questions.
On the negative side, we show that the separation principle indeed fails without additional
assumptions: explicit counterexamples are constructed already in the five-element modular
lattice $M_3$ and in a small finite orthomodular lattice (see
Propositions~\ref{prop:M3-counterexample} and~\ref{prop:OML-counterexample}).
On the positive side, we identify isotonicity of the lower bounding map $\Phi$ as a sufficient
replacement for distributivity: under the elementary condition
\[
 z\le t \quad\Longrightarrow\quad \Phi(z)\le \Phi(t),
\]
the same envelope construction yields a join homomorphism in \emph{arbitrary} lattices
(Theorem~\ref{thm:sep-isotone}).

The second main theme of the paper concerns orthomodular lattices, which arise naturally
as algebraic models of quantum logic.
These lattices are typically non-distributive, but they possess a rich internal structure:
every orthomodular lattice is the union of its Boolean blocks, i.e.\ its maximal Boolean
subalgebras.
Two elements are compatible precisely when they lie in a common block, and within each block
the lattice operations are fully distributive.
This suggests a \emph{blockwise} approach to stability:
one proves exact homomorphic selection results on each Boolean block and then studies when
such local selections can be assembled into a global map.

We develop this approach in detail.
For join-admissible (respectively meet-admissible) neighbourhood systems, approximate
identities on compatible pairs yield exact join (respectively meet) homomorphisms on each
Boolean block (Theorems~\ref{thm:blockwise-join} and~\ref{thm:blockwise-meet}).
When both joins and meets are approximately preserved and the neighbourhood system is
bi-admissible, we obtain genuine lattice homomorphisms on each block
(Theorem~\ref{thm:blockwise-lattice}).
We then investigate gluing problems for these blockwise selections.
Several positive gluing criteria are established, including a rigid-core condition and an
exact quotient-gluing mechanism for congruence neighbourhoods.
At the same time, we show that gluing can fail in a very concrete way when block overlaps are
non-trivial, even in finite orthomodular lattices
(Example~\ref{ex:glue-failure}).

Finally, we return to the Kalton--Roberts setting.
An approximately additive function on an orthomodular lattice is shown to be close, on each
Boolean block, to a finitely additive measure, with uniform control inherited from the
Boolean case (Theorem~\ref{thm:blockwise-states}).
This yields a natural blockwise extension of the Kalton--Roberts theorem to quantum logics,
illustrated by the projection lattice of a finite-dimensional Hilbert space.

The paper is organised as follows.
Section~2 collects preliminaries.
Section~3 contains the separation lemmas and their isotone variant.
Section~4 presents the modular and orthomodular counterexamples.
Sections~5--7 develop the neighbourhood framework, blockwise selection, and gluing theory.
Section~8 discusses bi-admissible selection and lattice homomorphisms.
Section~9 treats blockwise lattice selection.
Sections~10--11 discuss a partial converse and uniqueness.
Section~12 treats applications to approximate states and the Kalton--Roberts theorem,
Section~13 gives quotient gluing, and Section~14 lists open problems.

\section{Preliminaries}

\subsection{Lattices and conditional completeness}

\begin{definition}
A \emph{lattice} is a set $X$ with binary operations $\join$ and $\meet$ that are associative, commutative, idempotent, and satisfy absorption:
\[
x\join(x\meet y)=x,\qquad x\meet(x\join y)=x.
\]
The induced order is given by $x\le y$ iff $x\join y=y$ (equivalently $x\meet y=x$).
\end{definition}

\begin{definition}
A lattice $Y$ is \emph{conditionally complete} if every nonempty subset that is bounded above has a supremum and every nonempty subset bounded below has an infimum.
A \emph{complete Boolean algebra} is a Boolean algebra whose underlying order is complete.
\end{definition}

\subsection{Distributivity}

\begin{definition}
A lattice $X$ is \emph{distributive} if
\[
z\meet(x\join y)=(z\meet x)\join(z\meet y)\qquad(x,y,z\in X).
\]
\end{definition}

A basic consequence is that whenever $z\le x\join y$ in a distributive lattice, one has
\begin{equation}\label{eq:decomp}
z=z\meet(x\join y)=(z\meet x)\join(z\meet y).
\end{equation}

\begin{definition}\label{def:inf-distrib}
Let $Y$ be a conditionally complete lattice.
We say that $Y$ is \emph{infinitarily distributive} if, for every $y\in Y$ and every nonempty subset $A\subseteq Y$ for which the relevant suprema or infima exist, the following identities hold:
\[
y\meet\sup A=\sup\{y\meet a:a\in A\},
\qquad
 y\join\inf A=\inf\{y\join a:a\in A\}.
\]
Equivalently, finite meets distribute over existing nonempty bounded joins, and finite joins distribute over existing nonempty bounded meets.
This is the infinitary distributivity assumption used below in the applications of the neighbourhood-selection theorem of~\cite{BKP}.
\end{definition}

\subsection{Orthomodular lattices, blocks, compatibility}

\begin{definition}
An \emph{ortholattice} is a bounded lattice $(L,\join,\meet,0,1)$ with an order-reversing involution $x\mapsto x^\perp$
satisfying $x\join x^\perp=1$ and $x\meet x^\perp=0$.
An ortholattice is \emph{orthomodular} if $x\le y$ implies $y=x\join(y\meet x^\perp)$.
\end{definition}

\begin{definition}
A \emph{Boolean block} (or \emph{block}) in an orthomodular lattice is a maximal Boolean subalgebra.
Elements $x,y$ are \emph{compatible} if they lie in a common block.
They are \emph{orthogonal}, written $x\perp y$, if $x\le y^\perp$.
\end{definition}

\section{Separation lemmas}

\subsection{Join-sub/superhomomorphisms}
In order to formulate the separation principle, we distinguish between maps that preserve
joins exactly and those that preserve them only in one direction.
These notions capture the idea of approximate join preservation in a purely order-theoretic form.

\begin{definition}
Let $X,Y$ be lattices.
A map $F\colon X\to Y$ is a \emph{join homomorphism} if $F(x\join y)=F(x)\join F(y)$ for all $x,y\in X$.
A map $\Phi\colon X\to Y$ is a \emph{join-subhomomorphism} if $\Phi(x\join y)\le \Phi(x)\join\Phi(y)$.
A map $\Psi$ is a \emph{join-superhomomorphism} if $\Psi(x\join y)\ge \Psi(x)\join\Psi(y)$.
\end{definition}
The separation problem can now be stated as follows: given a join-subhomomorphism lying below
a join-superhomomorphism, can one interpolate a genuine join homomorphism between them?
In distributive lattices, this question has a positive answer.

\subsection{The distributive join separation lemma}
We begin with the separation result of Badora, Kochanek, and Przebieracz, which lies at the
heart of the neighbourhood-based stability theory for distributive lattices.
For completeness and to isolate the role of distributivity, we include a full proof.

\begin{theorem}[Join separation on distributive lattices {\cite[Lemma~6]{BKP}}]\label{thm:join-sep-distrib}
Let $X$ be a distributive lattice and $Y$ a conditionally complete lattice.
Suppose $\Phi,\Psi\colon X\to Y$ satisfy $\Phi\le \Psi$, with $\Phi$ join-subhomomorphic and $\Psi$ join-superhomomorphic.
Then
\[
F(x):=\sup\{\Phi(z): z\le x\}\qquad(x\in X)
\]
defines a join homomorphism $F\colon X\to Y$ such that $\Phi\le F\le \Psi$.
\end{theorem}

\begin{proof}
Fix $x\in X$.
Whenever $z\le x$, we have $\Phi(z)\le \Psi(z)$ and also $\Psi(z)\le \Psi(x)$, because $\Psi$ is join-superhomomorphic and $x=z\join x$ implies
\[
\Psi(x)=\Psi(z\join x)\ge \Psi(z)\join\Psi(x),
\]
hence $\Psi(z)\le \Psi(x)$.
Thus $\Phi(z)\le \Psi(x)$ for all $z\le x$, so the set $\{\Phi(z): z\le x\}$ is bounded above and its supremum exists.
The definition of $F$ therefore makes sense and satisfies $\Phi(x)\le F(x)\le \Psi(x)$.

The map $F$ is isotone: if $x\le y$ then $\{z:z\le x\}\subseteq\{z:z\le y\}$, so $F(x)\le F(y)$.
In particular, $F(x)\le F(x\join y)$ and $F(y)\le F(x\join y)$, whence
\begin{equation}\label{eq:easy}
F(x)\join F(y)\le F(x\join y).
\end{equation}

For the reverse inequality, fix $x,y\in X$ and consider any $z\le x\join y$.
Distributivity gives the decomposition \eqref{eq:decomp}:
$z=(z\meet x)\join(z\meet y)$.
Using join-subhomomorphy of $\Phi$,
\[
\Phi(z)\le \Phi(z\meet x)\join \Phi(z\meet y).
\]
Since $z\meet x\le x$ and $z\meet y\le y$, the definition of $F$ gives
$\Phi(z\meet x)\le F(x)$ and $\Phi(z\meet y)\le F(y)$.
Hence $\Phi(z)\le F(x)\join F(y)$ for every $z\le x\join y$.
Taking the supremum over $z\le x\join y$ gives $F(x\join y)\le F(x)\join F(y)$, which together with \eqref{eq:easy} yields
$F(x\join y)=F(x)\join F(y)$.
\end{proof}

\begin{remark}\label{rem:where-distrib}
In the proof of Theorem~\ref{thm:join-sep-distrib}, distributivity enters only through the identity $z=(z\meet x)\join(z\meet y)$ for $z\le x\join y$.
Everything else is purely lattice-theoretic and uses only conditional completeness of $Y$.
\end{remark}

\subsection{The dual meet separation lemma}
The previous subsection showed how a join-subhomomorphism lying below a join-superhomomorphism
can be separated by a join homomorphism.
We now record the dual statement for meets.

\begin{definition}
A map $\Psi\colon X\to Y$ is a \emph{meet-subhomomorphism} if $\Psi(x\meet y)\le \Psi(x)\meet\Psi(y)$.
A map $\Phi$ is a \emph{meet-superhomomorphism} if $\Phi(x\meet y)\ge \Phi(x)\meet\Phi(y)$.
\end{definition}

\begin{theorem}\label{thm:meet-sep-distrib}
Let $X$ be a distributive lattice and $Y$ a conditionally complete lattice.
Suppose $\Psi,\Phi\colon X\to Y$ satisfy $\Psi\le \Phi$, with $\Psi$ meet-subhomomorphic and $\Phi$ meet-superhomomorphic.
Then
\[
G(x):=\inf\{\Phi(z): z\ge x\}\qquad(x\in X)
\]
defines a meet homomorphism $G\colon X\to Y$ such that $\Psi\le G\le \Phi$.
\end{theorem}

\begin{proof}
Apply Theorem~\ref{thm:join-sep-distrib} to the order-dual lattices $X^{\op}$ and $Y^{\op}$.
The displayed formula is precisely the dual envelope, and the inequalities translate back to $\Psi\le G\le\Phi$.
\end{proof}

\subsection{Isotone separation without distributivity}
In Theorem~\ref{thm:join-sep-distrib} the only use of distributivity is the identity
$z=(z\meet x)\join(z\meet y)$ for $z\le x\join y$.
Outside the distributive setting this identity fails, but the argument still works if the lower bounding map is isotone.

\begin{theorem}[Join separation in arbitrary lattices under isotonicity]\label{thm:sep-isotone}
Let $X$ be a lattice and $Y$ a conditionally complete lattice.
Suppose $\Phi,\Psi\colon X\to Y$ satisfy $\Phi\le \Psi$, with $\Phi$ join-subhomomorphic and $\Psi$ join-superhomomorphic.
If
\[
 z\le t \quad\Longrightarrow\quad \Phi(z)\le \Phi(t),
\]
then the envelope $F(x)=\sup_{z\le x}\Phi(z)$ is a join homomorphism and satisfies $\Phi\le F\le \Psi$.
\end{theorem}

\begin{proof}
The definition of $F$ makes sense for the same reason as in Theorem~\ref{thm:join-sep-distrib}: for fixed $x$, the inequalities
$\Phi(z)\le \Psi(z)\le \Psi(x)$ for $z\le x$ show that $\{\Phi(z):z\le x\}$ is bounded above.
As before, one has $\Phi\le F\le \Psi$ and $F$ is isotone.

The inequality $F(x)\join F(y)\le F(x\join y)$ follows from isotonicity of $F$ exactly as in \eqref{eq:easy}.
For the reverse inequality, let $z\le x\join y$.
The isotonicity of $\Phi$ gives $\Phi(z)\le \Phi(x\join y)$, and join-subhomomorphy gives
\[
\Phi(x\join y)\le \Phi(x)\join\Phi(y)\le F(x)\join F(y).
\]
Thus $\Phi(z)\le F(x)\join F(y)$ for all $z\le x\join y$, and taking suprema over such $z$ yields $F(x\join y)\le F(x)\join F(y)$.
\end{proof}

\begin{remark}\label{rem:isotone-simple}
The hypothesis in Theorem~\ref{thm:sep-isotone} is exactly the standard isotonicity condition
$z\le t\Rightarrow \Phi(z)\le\Phi(t)$, and we use this simple formulation throughout.
The envelope
\[
F(x)=\sup\{\Phi(z):z\le x\}
\]
is isotone even when $\Phi$ is not: if $x\le y$ then $\{z:z\le x\}\subseteq\{z:z\le y\}$.
\end{remark}

\section{Obstructions: modular and orthomodular counterexamples}

\subsection{A finite orthomodular counterexample}

\begin{proposition}[Orthomodular obstruction]\label{prop:OML-counterexample}
There exist a finite orthomodular lattice $L$, the two-element Boolean algebra $Y$, and maps $\Phi,\Psi\colon L\to Y$ with $\Phi\le\Psi$
such that $\Phi$ is join-subhomomorphic and $\Psi$ join-superhomomorphic, but there is no join homomorphism $F\colon L\to Y$ with $\Phi\le F\le\Psi$.
\end{proposition}

\begin{proof}
Let $L$ be the horizontal sum of the two four-element Boolean algebras
\[
B_1=\{0,p,q,1\},\qquad B_2=\{0,r,s,1\},
\]
with $p^\perp=q$ and $r^\perp=s$.
Thus distinct atoms from different blocks have join $1$ and meet $0$.
Let $Y=\{0<1\}$ and define
\[
\begin{array}{c|cccccc}
x&0&p&q&r&s&1\\ \hline
\Phi(x)&0&0&0&1&0&0\\
\Psi(x)&0&0&0&1&0&1.
\end{array}
\]
The only nontrivial joins in $L$ are joins of distinct atoms, and these immediately give
$\Phi(u\join v)\le\Phi(u)\join\Phi(v)$ and
$\Psi(u)\join\Psi(v)\le\Psi(u\join v)$; hence $\Phi$ is join-subhomomorphic and $\Psi$ is join-superhomomorphic.

If $F$ were a join homomorphism with $\Phi\le F\le\Psi$, then $F(r)=1$ and therefore $F(1)=1$.
However, in the block $B_1$ we have $1=p\join q$, while the bounds force $F(p)=F(q)=0$.
Thus
\[
1=F(1)=F(p\join q)=F(p)\join F(q)=0,
\]
a contradiction.
\end{proof}

\subsection{Failure already in \texorpdfstring{$M_3$}{M3}}
The isotonicity assumption in Theorem~\ref{thm:sep-isotone} is essential.
Without either distributivity or isotonicity, the separation mechanism collapses,
and this failure occurs already in the five-element modular lattice $M_3$.

\begin{proposition}\label{prop:M3-counterexample}
Let $X=M_3=\{0,1,a,b,c\}$ be the five-element modular lattice and let $Y=\{0<1\}$.
There exist maps $\Phi,\Psi\colon X\to Y$ with $\Phi\le\Psi$ such that $\Phi$ is join-subhomomorphic and $\Psi$ join-superhomomorphic,
but no join homomorphism $F\colon X\to Y$ satisfies $\Phi\le F\le \Psi$.
\end{proposition}

\begin{proof}
The three atoms of $M_3$ satisfy $a\join b=a\join c=b\join c=1$.
Set
\[
\begin{array}{c|ccccc}
x&0&a&b&c&1\\ \hline
\Phi(x)&0&0&0&1&0\\
\Psi(x)&0&0&0&1&1.
\end{array}
\]
The atom-join relations show, as in Proposition~\ref{prop:OML-counterexample}, that $\Phi$ is join-subhomomorphic and $\Psi$ is join-superhomomorphic.
If $F$ were a join homomorphism with $\Phi\le F\le\Psi$, then $F(c)=1$, so $F(1)=1$.
But $1=a\join b$, whereas the bounds force $F(a)=F(b)=0$, a contradiction.
\end{proof}

\section{Neighbourhood systems}

\subsection{Join-, meet-, and bi-admissibility}
To formalise neighbourhood-based approximation, we recall the admissibility axioms
introduced in~\cite{BKP}.
Throughout the paper, a neighbourhood system $N$ is fixed globally on the codomain lattice $Y$; the same map $N\colon Y\to\Pow(Y)$ is used for all restrictions to Boolean blocks and for all gluing arguments.
These axioms are designed to ensure that approximate join or meet identities can be propagated through the lattice operations and converted into exact homomorphic selections.

\begin{definition}[Neighbourhood system]\label{def:neigh}
Let $Y$ be a lattice.
A map $N\colon Y\to\Pow(Y)$ assigns to each $y\in Y$ a nonempty \emph{bounded} subset $N(y)\subseteq Y$.
We say that $N$ is \emph{join-admissible} if for all $y,z,t,u\in Y$ the following hold:
\begin{enumerate}[label=\textup{(N\arabic*)},ref=N\arabic*]
\item\label{N1} $y\in N(y)$;
\item\label{N2} (\emph{convexity}) if $t,u\in N(z)$ and $t\le y\le u$, then $y\in N(z)$;
\item\label{N3} $\inf N(y)\in N(y)$ and $\sup N(y)\in N(y)$;
\item[\NFourJoin] (\emph{join stability}) if $t\in N(u)$ and $u\join y\in N(z)$, then $t\join y\in N(z)$.
\end{enumerate}
We say that $N$ is \emph{meet-admissible} if \NOne--\NThree\ hold and
\begin{enumerate}[label=\NFourMeet]
\item (\emph{meet stability}) if $t\in N(u)$ and $u\meet y\in N(z)$, then $t\meet y\in N(z)$.
\end{enumerate}
If $N$ is both join- and meet-admissible, we call it \emph{bi-admissible}.
\end{definition}

\subsection{Examples and a clarifying non-example}
The following examples show that the admissibility axioms are neither artificial nor
vacuous: they are satisfied by several natural constructions arising in lattice theory,
while at the same time excluding some seemingly reasonable but ultimately incompatible
choices.

\begin{example}[Principal ideals and filters are bi-admissible]\label{ex:ideals}
Assume that $Y$ has a least element $0$.
If $N_\downarrow(z):=\{y\in Y: y\le z\}$, then $N_\downarrow$ is bi-admissible.
Dually, if $Y$ has a greatest element $1$ and $N_\uparrow(z):=\{y\in Y: y\ge z\}$, then $N_\uparrow$ is bi-admissible.
\end{example}

\begin{proof}
We treat $N_\downarrow$.
The properties \NOne--\NThree\ are immediate.
For join stability \NFourJoin, assume $t\le u$ and $u\join y\le z$.
Then $t\join y\le u\join y\le z$, and so $t\join y\in N_\downarrow(z)$.
For meet stability \NFourMeet, assume $t\le u$ and $u\meet y\le z$.
Then $t\meet y\le u\meet y\le z$, so $t\meet y\in N_\downarrow(z)$.
\end{proof}

\begin{remark}[Rigid cores for principal ideals]\label{rem:rigid-ideal}
If $N=N_\downarrow$ is a principal-ideal system, then $N(f(x))$ is a singleton if and only if $f(x)=0$.
Consequently, the rigid-core condition ``$N(f(0))$ and $N(f(1))$ are singletons'' forces $f(0)=f(1)=0$.
This illustrates that rigid-core gluing is natural for congruence-type neighbourhoods (where many singletons may occur),
but can be restrictive for ideal-type neighbourhoods.
\end{remark}

\begin{example}[Congruence classes]\label{ex:congruence}
If $\Theta$ is a lattice congruence on $Y$ and $N_\Theta(y):=[y]_\Theta$, then $N_\Theta$ is bi-admissible; see \cite[Example~9(d)]{BKP}.
\end{example}

\begin{example}[Metric balls typically fail admissibility]\label{ex:metric-balls}
Let $Y=\mathbb{R}$ with $\join=\max$ and $\meet=\min$, and let $N_\varepsilon(z)=[z-\varepsilon,z+\varepsilon]$.
Then $N_\varepsilon$ satisfies \NOne--\NThree, but for every $\varepsilon>0$ it fails join stability \NFourJoin\ and meet stability \NFourMeet.
\end{example}

\begin{proof}
Fix $\varepsilon>0$ and take $z=0$, $u=\varepsilon$, $t=2\varepsilon$, $y=-1$.
Then $t\in N_\varepsilon(u)$ and $u\join y=\max\{\varepsilon,-1\}=\varepsilon\in N_\varepsilon(0)$,
but $t\join y=\max\{2\varepsilon,-1\}=2\varepsilon\notin[-\varepsilon,\varepsilon]=N_\varepsilon(0)$.
The meet case is dual.
\end{proof}

\begin{remark}
Although metric balls typically fail join-/meet-admissibility in the strict sense of
Definition~\ref{def:neigh}, they often satisfy a weaker, radius-inflating transport estimate.
For instance, in Example~\ref{ex:metric-balls} we have, for every $\varepsilon>0$ and $u,v\in\mathbb{R}$,
\[
N_\varepsilon(u)\ \join\ N_\varepsilon(v)\ \subseteq\ N_{2\varepsilon}(u\join v),
\qquad
N_\varepsilon(u)\ \meet\ N_\varepsilon(v)\ \subseteq\ N_{2\varepsilon}(u\meet v),
\]
where $N_\varepsilon(t)=[t-\varepsilon,t+\varepsilon]$ and $\join=\max$, $\meet=\min$.
Indeed, $\max$ and $\min$ are $1$-Lipschitz in each variable, so for
$u'\in N_\varepsilon(u)$ and $v'\in N_\varepsilon(v)$ one has
\[
\bigl|\max\{u',v'\}-\max\{u,v\}\bigr|\le |u'-u|+|v'-v|\le 2\varepsilon,
\]
and similarly for $\min$.
This highlights that the admissibility axioms are tuned to enable \emph{exact} algebraic
selection (true homomorphisms), whereas metric balls naturally lead to \emph{bounded-defect}
(quasi-)homomorphisms.
\end{remark}

\begin{remark}\label{rem:metric-vs-neigh}
Example~\ref{ex:metric-balls} explains why the neighbourhood formalism in \cite{BKP} does not coincide with naive metric-ball control.
If one wishes to work with metric balls, the error-function formulation of \cite[Theorem~7]{BKP} is the more natural framework,
and is, in practice, much closer to standard metric stability assumptions.
\end{remark}

\section{Orthomodular blockwise stability and gluing}

\subsection{Blockwise BKP selection for joins and meets}
In orthomodular lattices, distributivity is available only locally, on Boolean blocks.
The natural stability question is therefore not whether an approximate homomorphism admits
a global selection, but whether it admits exact homomorphic selections on each block.
The next result shows that join-admissibility is sufficient for such blockwise selections.

\begin{theorem}[Blockwise join selection on orthomodular lattices]\label{thm:blockwise-join}
Let $L$ be an orthomodular lattice and let $Y$ be an infinitarily distributive conditionally complete lattice in the sense of Definition~\ref{def:inf-distrib}.
Let $N\colon Y\to\Pow(Y)$ be a globally fixed join-admissible neighbourhood system.
If $f\colon L\to Y$ satisfies
\[
f(x)\join f(y)\in N\bigl(f(x\join y)\bigr)
\]
whenever $x$ and $y$ are compatible, then for every Boolean block $C\subseteq L$ there exists a join homomorphism
$F_C\colon C\to Y$ with $F_C(x)\in N(f(x))$ for all $x\in C$.
\end{theorem}

\begin{proof}
Fix a block $C$.
Since $C$ is a Boolean algebra, it is distributive, and every pair of elements in $C$ is compatible.
Thus the hypothesis restricts to the exact hypothesis of \cite[Theorem~8]{BKP} for the map $f|_C$ and the same neighbourhood system $N$ on $Y$.
That theorem produces a join homomorphism $F_C$ with $F_C(x)\in N(f(x))$ for all $x\in C$.
\end{proof}

\begin{theorem}[Blockwise meet selection on orthomodular lattices]\label{thm:blockwise-meet}
Let $L$ be an orthomodular lattice and let $Y$ be an infinitarily distributive conditionally complete distributive lattice.
Let $N\colon Y\to\Pow(Y)$ be a globally fixed meet-admissible neighbourhood system.
If $f\colon L\to Y$ satisfies
\[
f(x)\meet f(y)\in N\bigl(f(x\meet y)\bigr)
\]
for all compatible $x,y\in L$, then for every Boolean block $C\subseteq L$ there exists a meet homomorphism
$G_C\colon C\to Y$ with $G_C(x)\in N(f(x))$ for all $x\in C$.
\end{theorem}

\begin{proof}
This is the meet-dual of Theorem~\ref{thm:blockwise-join}, applied to $f|_C$.
Formally, one may either repeat the proof of \cite[Theorem~8]{BKP} with $\join$ and $\meet$ interchanged, or apply \cite{BKP} to the dual lattices.
\end{proof}

\subsection{Gluing compatible blockwise selections}
Having obtained exact homomorphic selections on each Boolean block, the next natural question
is whether these local maps can be assembled into a single global map.
The basic obstruction is purely set-theoretic: the selections must agree on overlaps.
When this compatibility condition is met, gluing is straightforward.

\begin{proposition}[Gluing]\label{prop:gluing}
Let $L$ be an orthomodular lattice and $Y$ a set.
Assume that for each block $C$ we have a map $F_C\colon C\to Y$, and that these maps agree on overlaps:
$F_C(x)=F_D(x)$ whenever $x\in C\cap D$.
Then there is a unique global map $F\colon L\to Y$ with $F|_C=F_C$ for every block $C$.

If $Y$ is a lattice and each $F_C$ is a join homomorphism, then $F$ preserves joins of compatible pairs.
If each $F_C$ is a meet homomorphism, then $F$ preserves meets of compatible pairs.
If each $F_C$ is a lattice homomorphism, then $F$ preserves both operations on compatible pairs.
\end{proposition}

\begin{proof}
Define $F(x)=F_C(x)$ for any block $C$ containing $x$.
The overlap hypothesis makes this well-defined, and uniqueness is clear.

If $x,y$ are compatible, they lie in some block $C$.
Since $C$ is a Boolean subalgebra, it is a sublattice of $L$; hence the join and meet computed in $C$ agree with those in $L$ for elements of $C$.
Thus, for example, if $F_C$ is a join homomorphism then
\[
F(x\join y)=F_C(x\join y)=F_C(x)\join F_C(y)=F(x)\join F(y),
\]
and the other cases are analogous.
\end{proof}

\begin{remark}
The blockwise viewpoint is closely related to the idea that an orthomodular lattice carries a `piecewise Boolean' structure,
with Boolean blocks playing the role of commeasurable contexts.
This perspective is developed in several directions by Heunen and collaborators; see \cite{HeunenPBA,HardingHeunenLN}.
\end{remark}

\begin{corollary}\label{cor:glue-01}
Assume the hypotheses of Theorem~\ref{thm:blockwise-join}, with the globally fixed neighbourhood system $N\colon Y\to\Pow(Y)$.
Suppose that any two distinct blocks of $L$ intersect only in $\{0,1\}$, and suppose that $N(f(0))$ and $N(f(1))$ are singletons.
Then there exists a global map $F\colon L\to Y$ such that $F|_C$ is a join homomorphism for every block $C$
and $F(x)\in N(f(x))$ for all $x\in L$.
In particular, $F$ preserves joins of compatible pairs.
\end{corollary}

\begin{proof}
Choose $F_C$ on each block $C$ using Theorem~\ref{thm:blockwise-join}.
If $C\ne D$, then $C\cap D=\{0,1\}$; the singleton hypothesis forces $F_C(0)=F_D(0)$ and $F_C(1)=F_D(1)$.
The overlap condition follows, and Proposition~\ref{prop:gluing} gives the global map.
\end{proof}

\begin{remark}
The intersection of all blocks of an orthomodular lattice is the centre; see, for example, \cite[Chapter~7]{Kalmbach} or \cite[Section~1.5]{PtP}.
This controls global overlap, but pairwise overlaps may be larger, and that is precisely where gluing can become delicate.
\end{remark}

\section{A concrete gluing failure example}
We now show that the existence of blockwise admissible homomorphic selections does not, by itself,
ensure the existence of a global selection.
The example below exhibits a finite orthomodular lattice in which blockwise join homomorphisms
exist and satisfy the neighbourhood constraints, yet may disagree on a non-trivial overlap.

\begin{example}\label{ex:glue-failure}
There exists a finite orthomodular lattice $L$ with two blocks $B_1,B_2$ such that $|B_1\cap B_2|=4$ and such that the following holds.
There are a distributive lattice $Y$, a globally fixed join-admissible neighbourhood system $N$ on $Y$, and a map $f\colon L\to Y$ satisfying the hypotheses of Theorem~\ref{thm:blockwise-join}; nevertheless, admissible blockwise join homomorphisms need not agree on $B_1\cap B_2$.
\end{example}

\begin{proof}
Let $B_1$ be the $8$-element Boolean algebra with atoms $e,a,b$, and let $B_2$ be the $8$-element Boolean algebra with atoms $e,c,d$.
Paste them along the common Boolean subalgebra generated by $e$,
\[
B_1\cap B_2=\{0,e,e^\perp,1\},
\]
and denote the resulting orthomodular lattice by $L=B_1\cup B_2$.

Let $Y=\{0<1<2\}$ and set
\[
N(0)=\{0\},\qquad N(1)=\{0,1,2\},\qquad N(2)=\{2\}.
\]
This $N$ is join-admissible: the case $N(z)=Y$ is automatic, and the singleton cases force both $u\join y$ and $t\join y$ to be the same singleton value.
Define $f(0)=0$ and $f(x)=1$ for every $x\ne0$.
Then $f(x)\join f(y)\in N(f(x\join y))$ for compatible $x,y$, since either $x=y=0$ or $N(f(x\join y))=N(1)=Y$.

On $B_1$ take $F_{B_1}\equiv0$.
On $B_2$ prescribe
\[
F_{B_2}(0)=0,
\qquad F_{B_2}(e)=2,
\qquad F_{B_2}(c)=F_{B_2}(d)=0,
\]
and extend by finite joins inside the Boolean algebra $B_2$.
Both maps are join homomorphisms and satisfy the required neighbourhood inclusions.
However, on the overlap one has
\[
F_{B_1}(e)=0\ne2=F_{B_2}(e),
\]
so these admissible blockwise selections cannot be glued.
\end{proof}

\section{Bi-admissible selection: from approximate \texorpdfstring{$\join$ and $\meet$}{joins and meets} to lattice homomorphisms}

\subsection{Extremal selections and induced sub/superhomomorphisms}
The admissibility axioms guarantee that neighbourhoods admit extremal elements.
Extracting these extremal points yields two canonical selections, one below and one above the
original map, which will turn out to satisfy the appropriate sub- and superhomomorphism
inequalities.

\begin{lemma}\label{lem:ab}
Let $Y$ be a conditionally complete lattice, let $N\colon Y\to\Pow(Y)$ satisfy \NOne--\NThree, and let $f\colon X\to Y$ be any map.
Define
\[
a(x)=\inf N(f(x)),\qquad b(x)=\sup N(f(x)).
\]
Then $a(x),b(x)\in N(f(x))$ for all $x$, and $a(x)\le f(x)\le b(x)$.
\end{lemma}

\begin{proof}
By \NThree, the infimum and supremum of $N(f(x))$ belong to $N(f(x))$.
Since $f(x)\in N(f(x))$ by \NOne, it lies between those extremal points, so $a(x)\le f(x)\le b(x)$.
\end{proof}

\begin{lemma}\label{lem:join-bounds}
Let $N$ be join-admissible and let $f\colon X\to Y$ satisfy $f(x)\join f(y)\in N(f(x\join y))$ for all $x,y\in X$.
With $a,b$ as in Lemma~\ref{lem:ab}, the map $a$ is join-subhomomorphic and $b$ is join-superhomomorphic.
\end{lemma}

\begin{proof}
Fix $x,y\in X$ and write $z=f(x\join y)$.
Since $a(x)\in N(f(x))$ and $f(x)\join f(y)\in N(z)$, join stability \NFourJoin\ gives $a(x)\join f(y)\in N(z)$.
Now $a(y)\in N(f(y))$, and another application of \NFourJoin\ yields $a(x)\join a(y)\in N(z)$.
By definition, $a(x\join y)=\inf N(z)$ is a lower bound of $N(z)$, so $a(x\join y)\le a(x)\join a(y)$.

The argument for $b$ is analogous: starting with $b(x)\in N(f(x))$ and then $b(y)\in N(f(y))$,
one obtains $b(x)\join b(y)\in N(z)$, and since $b(x\join y)=\sup N(z)$ is an upper bound of $N(z)$, one gets $b(x)\join b(y)\le b(x\join y)$.
\end{proof}
The corresponding statement for meets is entirely dual.

\begin{lemma}[Meet bounds]\label{lem:meet-bounds}
Let $N$ be meet-admissible and let $f\colon X\to Y$ satisfy $f(x)\meet f(y)\in N(f(x\meet y))$ for all $x,y\in X$.
With $a,b$ as in Lemma~\ref{lem:ab}, the map $a$ is meet-subhomomorphic and $b$ is meet-superhomomorphic.
\end{lemma}

\begin{proof}
The proof is the meet-dual of Lemma~\ref{lem:join-bounds}, using meet stability \NFourMeet\ in place of join stability \NFourJoin.
\end{proof}

\subsection{A bi-admissible selection theorem}
To combine the join and meet bounds obtained from the extremal selections, we require a
sandwich principle that interpolates between a join homomorphism lying below a meet
homomorphism.
The following lemma is a convenient reformulation of Kubi\'s sandwich theorem adapted to
this orientation.

\begin{lemma}[Kubi\'s sandwich theorem in the join-below-meet orientation]
\label{lem:kubis-join-below-meet}
Let $X$ be a distributive lattice and let $B$ be a complete Boolean algebra.
Suppose that $F,G\colon X\to B$ satisfy:
\begin{enumerate}[label=(\roman*)]
\item $F$ is a join homomorphism;
\item $G$ is a meet homomorphism;
\item $F(x)\le G(x)$ for every $x\in X$.
\end{enumerate}
Then there exists a lattice homomorphism $H\colon X\to B$ such that
\[
F(x)\le H(x)\le G(x)\qquad\text{for all }x\in X.
\]
\end{lemma}

\begin{proof}
This is a direct corollary of Kubi\'s sandwich theorem for $S_4$ bi-convexity spaces
\cite[Theorem~3.3]{Kubis}.
In particular, Kubi\'s lattice corollary \cite[Theorem~3.7]{Kubis} (\emph{cf}.\ also \cite[Theorem~4]{BKP})
covers the orientation ``meet homomorphism below join homomorphism''.
Since Kubi\'s theorem is symmetric under interchanging the two convexities (ideals and filters),
swapping them yields the present ``join below meet'' form, with the same conclusion.
\end{proof}
We are now in a position to combine the extremal selections and the sandwich principle into
a single selection theorem.
Under bi-admissibility, approximate preservation of both joins and meets forces the existence
of an exact lattice homomorphism taking values inside the prescribed neighbourhoods.

\begin{theorem}[Bi-admissible selection yields lattice homomorphisms]\label{thm:bi-admissible}
Let $X$ be a distributive lattice and let $B$ be a complete Boolean algebra.
Let $N\colon B\to\Pow(B)$ be bi-admissible.
Assume that $f\colon X\to B$ satisfies, for all $x,y\in X$,
\[
f(x)\join f(y)\in N\bigl(f(x\join y)\bigr),\qquad
f(x)\meet f(y)\in N\bigl(f(x\meet y)\bigr).
\]
Then there exists a lattice homomorphism $H\colon X\to B$ such that $H(x)\in N(f(x))$ for every $x\in X$.
\end{theorem}

\begin{proof}
Let $a,b$ be the extremal selections from Lemma~\ref{lem:ab}.
By Lemmas~\ref{lem:join-bounds} and \ref{lem:meet-bounds}, the map $a$ is simultaneously join-subhomomorphic and meet-subhomomorphic,
and $b$ is simultaneously join-superhomomorphic and meet-superhomomorphic.
The pointwise inequality $a\le b$ is immediate from the definitions.

Applying the distributive join separation theorem (Theorem~\ref{thm:join-sep-distrib}) to $\Phi=a$ and $\Psi=b$, one obtains a join homomorphism $F$ with $a\le F\le b$.
Applying the meet separation theorem (Theorem~\ref{thm:meet-sep-distrib}) to $\Psi=a$ and $\Phi=b$, one obtains a meet homomorphism $G$ with $a\le G\le b$.

It is convenient to note that meet-subhomomorphisms and join-superhomomorphisms are isotone.
For meet-subhomomorphisms this follows by duality from the corresponding observation for join-superhomomorphisms.
For join-superhomomorphisms, if $u\le v$ then $v=u\join v$, and hence $\Psi(v)=\Psi(u\join v)\ge \Psi(u)\join\Psi(v)$, forcing $\Psi(u)\le \Psi(v)$.

Now fix $x\in X$.
If $z\le x\le w$, then isotonicity gives $a(z)\le a(x)\le b(x)\le b(w)$.
Taking suprema over $z\le x$ yields $F(x)\le b(w)$ for all $w\ge x$, and hence $F(x)\le \inf_{w\ge x} b(w)=G(x)$.
Thus $F\le G$ pointwise.

Finally, Lemma~\ref{lem:kubis-join-below-meet} yields a lattice homomorphism $H\colon X\to B$
such that $F\le H\le G$.
In particular, for each $x\in X$ we have
\[
a(x)\le F(x)\le H(x)\le G(x)\le b(x).
\]
The points $a(x)$ and $b(x)$ belong to $N(f(x))$ by Lemma~\ref{lem:ab}, and $N(f(x))$ is convex by \NTwo,
so every element between them lies in $N(f(x))$.
Therefore $H(x)\in N(f(x))$ for all $x\in X$.
\end{proof}

\section{Orthomodular blockwise lattice selection}

\begin{theorem}[Blockwise lattice homomorphisms under bi-admissibility]\label{thm:blockwise-lattice}
Let $L$ be an orthomodular lattice and let $B$ be a complete Boolean algebra.
Let $N\colon B\to\Pow(B)$ be a globally fixed bi-admissible neighbourhood system.
Assume that $f\colon L\to B$ satisfies, for all compatible $x,y\in L$,
\[
f(x)\join f(y)\in N\bigl(f(x\join y)\bigr),\qquad
f(x)\meet f(y)\in N\bigl(f(x\meet y)\bigr).
\]
Then for every Boolean block $C\subseteq L$ there exists a lattice homomorphism $H_C\colon C\to B$ such that $H_C(x)\in N(f(x))$ for all $x\in C$.
\end{theorem}

\begin{proof}
Fix a block $C$.
Since $C$ is a Boolean algebra, it is distributive, and every pair of elements in $C$ is compatible.
The hypotheses therefore restrict to the hypotheses of Theorem~\ref{thm:bi-admissible} for $f|_C\colon C\to B$ and the same neighbourhood system $N$.
Applying that theorem gives the required lattice homomorphism $H_C$.
\end{proof}

\section{A partial converse and the necessity of symmetry}
The selection results obtained so far are one-directional: approximate identities lead to
exact homomorphic selections.
It is natural to ask to what extent the converse holds.
In this section we isolate a minimal additional assumption on the neighbourhood system
under which exact selections force the corresponding approximate identities.

\begin{definition}\label{def:symmetric}
A neighbourhood system $N\colon Y\to\Pow(Y)$ is \emph{symmetric} if $t\in N(u)$ implies $u\in N(t)$.
Congruence-class neighbourhoods are symmetric.
\end{definition}
Symmetry allows one to reverse neighbourhood membership, turning pointwise control of an
exact homomorphism into an approximate identity for the original map.
This yields a partial converse to the selection results obtained earlier.

\begin{proposition}\label{prop:partial-converse}
Let $X,Y$ be lattices and let $N\colon Y\to\Pow(Y)$ be symmetric.

If $N$ is join-admissible and $H\colon X\to Y$ is a join homomorphism such that $H(x)\in N(f(x))$ for all $x\in X$,
then $f$ satisfies the join-approximation condition
\[
f(x)\join f(y)\in N\bigl(f(x\join y)\bigr)\qquad(x,y\in X).
\]
The dual statement holds for meets under meet-admissibility.
If $N$ is bi-admissible and $H$ is a lattice homomorphism with $H(x)\in N(f(x))$, then both approximate conditions hold.
\end{proposition}

\begin{proof}
We prove the join statement; the meet statement is dual.
Fix $x,y\in X$ and set $z=f(x\join y)$.
Since $H$ is a join homomorphism, $H(x)\join H(y)=H(x\join y)$.
By hypothesis, $H(x\join y)\in N(f(x\join y))=N(z)$.

Because $H(x)\in N(f(x))$ and $N$ is symmetric, we also have $f(x)\in N(H(x))$.
Join stability \NFourJoin\ applied with $t=f(x)$ and $u=H(x)$ shows that
$f(x)\join H(y)\in N(z)$.
Similarly, symmetry gives $f(y)\in N(H(y))$, and another application of \NFourJoin\ yields
$f(x)\join f(y)\in N(z)$.
\end{proof}

\begin{example}[Symmetry is genuinely needed]\label{ex:symmetry-needed}
There exist a join-admissible neighbourhood system $N$ that is not symmetric, and maps $f,H\colon X\to Y$ such that
$H$ is a join homomorphism with $H(x)\in N(f(x))$ for all $x$, but $f$ fails the join-approximation condition.
\end{example}

\begin{proof}
Let $Y=\{0<1<2\}$ and define
\[
N(0)=\{0\},\qquad N(1)=\{1\},\qquad N(2)=\{1,2\}.
\]
This $N$ is join-admissible, but it is not symmetric because $1\in N(2)$ while $2\notin N(1)$.

Let $X=\{0,1\}$ be the two-element lattice.
Define $f(0)=2$ and $f(1)=1$.
Define $H\equiv 1$ on $X$; this is a join homomorphism, and $H(0)=1\in N(2)=N(f(0))$, while $H(1)=1\in N(1)=N(f(1))$.
However,
\[
f(0)\join f(1)=2\notin N(1)=N\bigl(f(0\join 1)\bigr),
\]
so the join-approximation condition fails.
The obstruction is exactly that one cannot reverse membership $H(x)\in N(f(x))$ to $f(x)\in N(H(x))$ without symmetry.
\end{proof}

\section{Uniqueness via generators}
Although the neighbourhood framework naturally emphasises existence, uniqueness questions
arise whenever neighbourhoods degenerate to points.
The following definitions and results explain how uniqueness can be recovered from finite
join-generation.

\begin{definition}
A subset $J\subseteq X$ \emph{finite join-generates} a lattice $X$ if every $x\in X$ can be written as a finite join of elements of $J$.
\end{definition}

\begin{proposition}\label{prop:uniq}
Let $X$ be a lattice and $Y$ a lattice.
If $J\subseteq X$ finite join-generates $X$ and two join homomorphisms $H_1,H_2\colon X\to Y$ agree on $J$ (and on $0$ when it is realised as the empty join),
then $H_1=H_2$ on $X$.
\end{proposition}

\begin{proof}
Fix $x\in X$ and choose $x=j_1\join\cdots\join j_n$ with $j_k\in J$.
Then
\[
H_1(x)=H_1(j_1)\join\cdots\join H_1(j_n)=H_2(j_1)\join\cdots\join H_2(j_n)=H_2(x).
\]
\end{proof}

\begin{corollary}\label{cor:unique}
In the setting of Theorem~\ref{thm:bi-admissible}, assume $X$ has a finite join-generating set $J$ such that each neighbourhood $N(f(j))$ is a singleton for $j\in J$
(and for $0$ if needed).
Then the lattice homomorphism $H$ with $H(x)\in N(f(x))$ is unique.
\end{corollary}

\begin{proof}
Any admissible selection must take the unique value in $N(f(j))$ at each $j\in J$.
Thus any two such homomorphisms agree on $J$, and Proposition~\ref{prop:uniq} gives uniqueness.
\end{proof}

\begin{remark}
Suppose that $X$ is a finite distributive lattice.
Let $\mathrm{Ji}(X)$ denote the set of join-irreducible elements of $X$, i.e.\ the set of all $j\ne 0$ such that
$j=x\join y$ implies $j=x$ or $j=y$.
Then $\mathrm{Ji}(X)$ finite join-generates $X$ (indeed, every element of $X$ is the join of the join-irreducibles below it).
Consequently, Corollary~\ref{cor:unique} applies with $J=\mathrm{Ji}(X)$ (and $0$ viewed as the empty join):
if each $N(f(j))$ is a singleton for $j\in \mathrm{Ji}(X)$ (and for $0$ if needed), then the admissible lattice homomorphism is unique.
\end{remark}

\section{Approximate states on orthomodular lattices and Kalton--Roberts}

\subsection{Blockwise approximation}
We now turn to applications of the blockwise framework to additive functions on
orthomodular lattices.
In this setting, the natural notion of additivity is not global, but restricted to
orthogonal (hence compatible) pairs, reflecting the underlying block structure.

\begin{definition}
Let $L$ be an orthomodular lattice.
A map $\mu\colon L\to\mathbb{R}$ is \emph{finitely additive on orthogonal pairs} if $\mu(x\join y)=\mu(x)+\mu(y)$ whenever $x\perp y$.
If additionally $\mu(0)=0$, $\mu(1)=1$, and $\mu\ge 0$, then $\mu$ is a (finitely additive) \emph{state}.
\end{definition}
On a Boolean block, orthogonality is simply disjointness, so the restriction of $\mu$
to any block is a nearly additive set function in the sense of Kalton and Roberts.
This observation underlies the blockwise approximation result proved below.

\begin{theorem}[Kalton--Roberts]\label{thm:KR}
Let $B$ be a Boolean algebra and let $f\colon B\to\mathbb{R}$ satisfy
\[
\bigl|f(x\join y)-f(x)-f(y)\bigr|\le \varepsilon\qquad(x\meet y=0).
\]
Then there exists a finitely additive signed measure $\nu$ on $B$ such that 
\[
    \sup_{b\in B}|f(b)-\nu(b)|\le K\,\varepsilon,
\]
where $K$ is a universal constant.
\end{theorem}

\begin{proof}
Apply the case $\varepsilon=1$ in \cite{KR} (or the refinement \cite{BPR}) to the rescaled function $f/\varepsilon$, and scale back.
\end{proof}

\begin{remark}
The optimal value of $K$ is unknown. One may take $K=44.5$ by \cite{KR,BKP}, and $K=38.5$ by \cite{BPR}. A non-trivial lower bound is provided by Gnacik, Guzik, and the author \cite{GGK}, who show that $K\ge 3$ already for non-negative $1$-additive functions.
\end{remark}

\begin{theorem}[Blockwise approximation of orthomodular states]\label{thm:blockwise-states}
Let $L$ be an orthomodular lattice and let $\mu\colon L\to\mathbb{R}$ satisfy
\[
\bigl|\mu(x\join y)-\mu(x)-\mu(y)\bigr|\le \varepsilon\qquad(x\perp y).
\]
Then for every Boolean block $C\subseteq L$ there exists a finitely additive map $\nu_C\colon C\to\mathbb{R}$ such that
\[
\sup_{c\in C}|\mu(c)-\nu_C(c)|\le K\,\varepsilon.
\]
\end{theorem}

\begin{proof}
Fix a block $C$.
In $C$ orthogonality is simply disjointness ($x\meet y=0$), so the hypothesis restricts to the Kalton--Roberts hypothesis on $C$.
Apply Theorem~\ref{thm:KR} to $\mu|_C$.
\end{proof}

\begin{remark}
The content of the theorem is not in the proof, which is a direct restriction argument,
but in the identification of orthogonal additivity as the correct notion allowing
Kalton--Roberts stability to be transferred from Boolean algebras to orthomodular lattices.
\end{remark}

\begin{remark}[Orthogonality implies a common block]
If $x\perp y$ in an orthomodular lattice, then $x\le y^\perp$.
The elements $x,y,y^\perp$ generate a Boolean subalgebra, which is contained in some block; see \cite[Chapter~2]{Kalmbach}.
This is the basic reason why orthogonal additivity is naturally a blockwise (hence Boolean) condition.
\end{remark}

\subsection{A finite-dimensional projection example}

\begin{example}\label{ex:proj}
Let $H=\mathbb{C}^n$ and let $P(H)$ be the projection lattice, an orthomodular lattice under range inclusion with $p^\perp=I-p$.
Fix a density matrix $\rho$ and let $\mu_0(p)=\mathrm{tr}(\rho p)$.
If $\eta\colon P(H)\to\mathbb{R}$ satisfies $\|\eta\|_\infty\le 1$ and $\mu(p)=\mu_0(p)+\varepsilon\,\eta(p)$,
then for orthogonal projections $p\perp q$ one has
\[
|\mu(p\join q)-\mu(p)-\mu(q)|=\varepsilon\,|\eta(p\join q)-\eta(p)-\eta(q)|\le 3\varepsilon.
\]
Hence Theorem~\ref{thm:blockwise-states} yields, on each block of commuting projections, a finitely additive approximation within $3K\varepsilon$ in sup norm.
In the block determined by an orthonormal basis $(e_i)_{i=1}^n$, the projections are precisely those onto coordinate subspaces,
\[
p_A=\sum_{i\in A} e_i\otimes e_i^*,\qquad A\subseteq\{1,\dots,n\}.
\]
\end{example}

\section{Exact quotient gluing for congruence neighbourhoods}
The rigid-core condition guarantees exact agreement on overlaps, but it is not the only way
to enforce gluing.
When neighbourhoods arise from lattice congruences, it is natural to pass to the quotient,
where compatibility on overlaps becomes automatic.

\begin{proposition}[Quotient gluing]\label{prop:quotient-glue}
Let $L$ be an orthomodular lattice and $Y$ a lattice.
Let $\Theta$ be a lattice congruence on $Y$ and let $N(y)=[y]_\Theta$.
Suppose that for each block $C\subseteq L$ we have a map $F_C\colon C\to Y$ such that $F_C(x)\in [f(x)]_\Theta$ for all $x\in C$.
Then the compositions $\overline{F}_C\colon C\to Y/\Theta$ agree on overlaps and therefore glue to a global map $\overline{F}\colon L\to Y/\Theta$.
Moreover, $\overline{F}(x)=[f(x)]_\Theta$ for all $x\in L$.
\end{proposition}

\begin{proof}
If $x\in C\cap D$, then $F_C(x)\equiv f(x)\equiv F_D(x)\ (\Theta)$, so $\overline{F}_C(x)=\overline{F}_D(x)$ in $Y/\Theta$.
Thus the quotient-valued block maps agree on every overlap, and Proposition~\ref{prop:gluing} gives the global quotient map.
The last statement is immediate from $\overline{F}(x)=[F_C(x)]_\Theta=[f(x)]_\Theta$.
\end{proof}

\section{Open problems}

\begin{problem}\label{prob:quant-glue}
Example~\ref{ex:glue-failure} shows that exact gluing can fail when block overlaps are non-trivial.
Suppose, however, that one has \emph{quantitative} overlap control: for some metric (or lattice seminorm) on $Y$,
the blockwise selections satisfy $d(F_C(x),F_D(x))\le \delta$ for all $x\in C\cap D$.
Under what additional hypotheses can one construct a global map whose defect is controlled in terms of $\delta$?
\end{problem}

\begin{problem}\label{prob:MV}
Orthomodular lattices have Boolean blocks, and the same is true for orthoalgebras.
The question below is therefore not meant to concern orthoalgebras.
It concerns only those many-valued or lattice-ordered generalisations in which MV-algebraic contexts are genuinely present.
In particular, one should not expect such a formulation for arbitrary effect algebras: if an effect algebra is not a lattice, it need not be a union of maximal MV-subalgebras.

For classes in which MV-blocks or MV-contexts are part of the structure, is there an MV-algebraic analogue of the separation and neighbourhood-selection machinery
developed in this paper, yielding stability results for approximately additive maps on
MV-blocks and a suitable blockwise gluing theory?
At present we do not know a general answer.

It would also be interesting to understand $\sigma$-additive refinements in settings where
the relevant blocks are $\sigma$-complete Boolean algebras (for instance, in projection
lattices of separable Hilbert spaces), and to what extent $\sigma$-additive
Kalton--Roberts-type stability can be proved blockwise and then glued.
\end{problem}

\section*{Acknowledgments}
The author thanks the referee for a careful reading of the manuscript and for constructive comments.
These comments led, in particular, to a simplification of Proposition~\ref{prop:OML-counterexample}, a shorter treatment of the dual separation theorem, and clearer formulations of the neighbourhood system and infinitary distributivity assumptions.
Support from NCN Sonata-Bis 13, grant no. 2023/50/E/ST1/00067 is acknowledged with thanks.

\section*{Conflict of Interest and Data Availability}

\paragraph{Conflict of Interest.}
The author declares that he has no conflict of interest.

\paragraph{Data Availability.}
No datasets were generated or analysed during the current study.


\end{document}